\documentclass[12pt]{amsart}
\usepackage{amsmath, amsthm, amscd, amsfonts, amssymb, graphicx, color,float,pgf,tikz, mathrsfs}
\usepackage{amssymb,fontenc}
\usepackage{latexsym,wasysym,mathrsfs}
\usepackage{hyperref}
\usepackage{booktabs} 
\usepackage{array}
\usepackage[utf8]{inputenc}
\usepackage[T1]{fontenc}
\usepackage{booktabs}
\usepackage{geometry}
\usepackage{subcaption}
\usepackage{soul}

\newtheorem{theorem}{Theorem}[section]
\newtheorem{lemma}[theorem]{Lemma}
\newtheorem{proposition}[theorem]{Proposition}
\newtheorem{corollary}[theorem]{Corollary}

\theoremstyle{definition}
\newtheorem{definition}[theorem]{Definition}

\newtheorem{remark}[theorem]{Remark}

\def\stirling2#1#2{\genfrac{\{}{\}}{0pt}{}{#1}{#2}}

\newcommand{\vn}{\mathbf{1}^n}

\title{Arithmetic Properties of Mixed Stirling Numbers of the second kind}
\author{Daniel Yaqubi$^{*}$}
\address{\bf $^*$Department of Computer science, University of Torbat-e Jam, Torbat-e Jam, Iran.}
\email{yaqubi@tjamcaas.ac.ir, or daniel\_yaqubi@yahoo.es}
\author{Madjid Mirzavaziri$^{**}$}
\address{\bf Department of Pure Mathematics, Faculty of Mathematical Sciences,  Ferdowsi University of Mashhad, P. O. Box 1159-91775, Mashhad, Iran.}
\email{mirzavaziri@gmail.com}


\begin{document}

\begin{abstract}
We investigate the mixed Stirling numbers of the second kind, $\mathcal{S}(n; \mathbf{c})$, which count partitions of $n$ distinct elements into $m$ unlabeled and $k$ labeled non-empty blocks encoded by $\mathbf{c} = (m, 1^k)$. We establish their recurrence relations and exponential generating functions, and analyze their behavior modulo a prime $p$ and $p^2$. In particular, we extend the classical Touchard congruence to this mixed framework. Our results show that these configurations possess unique number-theoretic signatures distinct from classical set partitions.
\end{abstract}
\maketitle
\section{Introduction}

The classical problem of partitioning a set of $n$ elements into $k$ non-empty blocks is counted by the Stirling numbers of the second kind, denoted by $\left\{ \genfrac{}{}{0pt}{}{n}{k} \right\}$. We extend this to a general occupancy problem where both balls and cells possess multiplicities. The authors \cite{Yaqubi2016, Yaqubi2019, Yaqubi2020} introduced a generalization of this problem as \textit{Mixed Partitions}. In this generalization, both balls and cells are as multi-sets. Let $\mathbf{b}=(b_1, \dots, b_n)$ and $\mathbf{c}=(c_1, \dots, c_k)$ represent ball and cell multiplicities, respectively. Authors \cite{Yaqubi2016}, investigated the fundamental properties of the number of ways to partition a multi-set of balls $\mathbf{b}=(b_1, \dots, b_n)$ into a multi-set of non-empty cells $\mathbf{c}=(c_1, \dots, c_k)$ and found several results for special cases.

The \textit{mixed (Multi-set) Stirling number} is denoted by $\mathcal{S}(\mathbf{b}; \mathbf{c})$, represents the number of ways to partition a multi-set of balls with multiplicities $\mathbf{b}$ into a multi-set of non-empty cells with multiplicities $\mathbf{c}$. For the multi-set of balls  $\mathbf{b}$, \textit{the mixed Bell number} $\mathbf{b}$ is defined as
\begin{equation}
    \mathcal{B}(\mathbf{b})
    =
    \sum_{\mathbf{c}}
    \mathcal{S}(\mathbf{b};\mathbf{c}),
\end{equation}
where the sum ranges over all admissible nonzero cell multiplicity vectors $\mathbf{c}$.
In special case, for $n \ge 0$ and fixed $k \ge 0$, the mixed Bell number of order $k$ is defined as
\begin{equation}
    \mathcal{MB}_{n,k}=\sum_{m=0}^{n-k}\mathcal{S}(\mathbf{1}^n;(m,1^k)).
\end{equation}
In this paper, we study the arithmetic properties of the mixed Stirling numbers of the second kind $\mathcal{S}(n; \mathbf{c})$. This number counts the partitions of $n$ distinct elements into mixed blocks $\mathbf{c} = (m, 1^k)$, which represents $m$ unlabeled non-empty blocks and $k$ labeled non-empty blocks.

The main focus of our work is to analyze how these mixed Stirling numbers behave modulo a prime $p$ and its higher powers like $p^2$. In particular, we extend the classical Touchard congruence,
\[
B_{p+n} \equiv B_n + B_{n+1} \pmod p,
\]
to our mixed framework. By applying the Frobenius differential operator $D^p \equiv D \pmod p$ to the exponential generating functions, we show how the configuration vector $\mathbf{c}$ controls the periodicity of $\mathcal{S}(n; \mathbf{c})$ in modular arithmetic.
\section{Primary results}
We start this section with some fundamental results about mixed stirling numbers $\mathcal{S}(\vn; \mathbf{c})$. 
\begin{theorem}\label{Fun}
For $n \ge 1$ and a cell vector $\mathbf{c} = (c_1, \dots, c_k)$, we have
\begin{equation}
    \mathcal{S}(\vn; \mathbf{c})=\left( \sum_{i=1}^k c_i \right) \mathcal{S}(\mathbf{1}^{n-1}; \mathbf{c}) + \sum_{i=1}^k \mathcal{S}(\mathbf{1}^{n-1}; \mathbf{c} - \mathbf{e}_i)
\end{equation}
where $\mathbf{e}_i$ is the $i$-th standard basis vector.
\end{theorem}
\begin{proof}
We prove the recurrence by considering the placement of the $n$-th ball in a partition of $n$ distinct balls into the multi-set of cells $\mathbf{c}$. There are two cases.
\begin{itemize}
    \item \textbf{Case 1: The $n$-th ball not comes as alone in any cell.} \\
    In this scenario, the first $n-1$ balls must have already been partitioned into the full multi-set of cells $\mathbf{c}$ such that no cell is empty. This is counted by $\mathcal{S}(\mathbf{1}^{n-1}; \mathbf{c})$. The $n$-th ball can then be placed into any of the $K = \sum_{i=1}^k c_i$ available cells. This contributes
$ \left( \sum_{i=1}^k c_i \right) \mathcal{S}(\mathbf{1}^{n-1}; \mathbf{c}).$ 
   \item \textbf{Case 2: The $n$-th ball comes as alone in any cell.} \\
    In this scenario, the $n$-th ball construct a new cell of some type $i \in \{1, \dots, k\}$. Consequently, the remaining $n-1$ balls must have been partitioned into the reduced multiset of cells $\mathbf{c} - \mathbf{e}_i$ such that all those cells were non-empty. For each possible type $i$ that the $n$-th ball could represent, there are $\mathcal{S}(\mathbf{1}^{n-1}; \mathbf{c} - \mathbf{e}_i)$ such ways. Summing over all possible types $i$, we obtain $\sum_{i=1}^k \mathcal{S}(\mathbf{1}^{n-1}; \mathbf{c} - \mathbf{e}_i).$
    \end{itemize}
Combining these two cases conclude the proof.
\end{proof}
In Theorem \ref{Fun}, when $k=1$, we have
 $\mathcal{MB}_{n,0} = \sum_{m=0}^n \mathcal{S}(\mathbf{1}^n;(m)) = B_n,$
 so the mixed Bell numbers reduce to the classical Bell numbers.
Furthermore, in the specific case $\mathbf{c} = (k)$, the multiset consists of a single type of cell with multiplicity $k$. Then
\begin{equation*}
    \mathcal{S}(\mathbf{1}^n; k) = k \mathcal{S}(\mathbf{1}^{n-1}; k) + \mathcal{S}(\mathbf{1}^{n-1}; k-1),
\end{equation*}
which is the classical recurrence for the Stirling numbers of the second kind. The authors~\cite{Yaqubi2019}, using the symbolic method of Flajolet and Sedgewick~\cite{flajolet1980}, obtained the following exponential generating function.
\begin{lemma}[\cite{Yaqubi2019}]
\label{Gef}
The exponential generating function (EGF) for the multi-set Stirling number $\mathcal{S}(n; \mathbf{c})$ is given by
\begin{equation*}
    \mathcal{E}(x) = \sum_{n \ge 0} \mathcal{S}(n; \mathbf{c}) \frac{x^n}{n!} = \prod_{i=1}^k \frac{(e^x - 1)^{c_i}}{c_i!}.
\end{equation*}
\end{lemma}
The case $\mathcal{S}(\mathbf{1}^n;\mathbf{c})$ admit a representation in terms of the exponential partial Bell polynomials $B_{n,K}(x_1,x_2,\dots)$. Let
$K=\sum_{i=1}^k c_i.$ Since
$B_{n,K}(1!,2!,3!,\dots)
=
\left\{\genfrac{}{}{0pt}{}{n}{K}\right\},$ we obtain
\begin{equation*}
\mathcal{S}(\mathbf{1}^n;\mathbf{c})
=
\frac{K!}{\prod_{i=1}^k c_i!}
\,B_{n,K}(1!,2!,3!,\dots).
\end{equation*}
Equivalently,
$\mathcal{S}(\mathbf{1}^n;\mathbf{c})
=\frac{K!}{\prod_{i=1}^t c_i!}\left\{\genfrac{}{}{0pt}{}{n}{k}\right\}.$
This representation permits the transfer of structural properties of Bell polynomials, including recurrence relations, generating functions, and determinantal representations, to the study of multi-set partitions.
\begin{theorem}\label{Inclusion-Exclusion Formula}
For a cell multiplicity vector $\mathbf{c} = (c_1, \dots, c_k)$, the multi-set Stirling number is given by
\begin{equation*}
\mathcal{S}(\mathbf{1}^n; \mathbf{c}) = \frac{1}{\prod_{i=1}^k c_i!} \sum_{\mathbf{0} \le \mathbf{j} \le \mathbf{c}} (-1)^{|\mathbf{c}| - |\mathbf{j}|} \left( \prod_{i=1}^k \binom{c_i}{j_i} \right) \left( \sum_{m=1}^k j_m \right)^n
\end{equation*}
where $|\mathbf{c}| = \sum_{i=1}^k c_i$ and $|\mathbf{j}| = \sum j_i$.
\end{theorem}
\begin{proof}
Let $K = |\mathbf{c}| = \sum_{i=1}^{k} c_i$ denote the total number of cells. 
We first temporarily regard all $K$ cells as distinct, ignoring the fact that cells of the same type are indistinguishable. By the Principle of Inclusion-Exclusion, the number of surjective mappings from an $n$-element set onto $K$ labeled cells is $\sum_{i=0}^{K} (-1)^{K-i} \binom{K}{i} i^n.$ To recover the multiset partition numbers, we group the selected cells according to their types. Let $\mathbf{j} = (j_1, \ldots, j_k)$, where $0 \le j_i \le c_i$, denote the number of selected cells of each type. The number of ways to choose these specific cells is $\prod_{i=1}^{k} \binom{c_i}{j_i}, $ and the total number of selected cells is given by $|\mathbf{j}| = \sum_{i=1}^{k} j_i$. 
For a fixed choice of $\mathbf{j}$, there are $|\mathbf{j}|^n$ functions from the $n$ balls into the selected cells. Applying the Principle of Inclusion-Exclusion over the subset of active cells yields a contribution of
\begin{equation*}
    (-1)^{|\mathbf{c}| - |\mathbf{j}|} \left( \prod_{i=1}^{k} \binom{c_i}{j_i} \right) |\mathbf{j}|^n.
\end{equation*}
Summing over all possible vectors $\mathbf{0} \le \mathbf{j} \le \mathbf{c}$ accounts for all valid configurations of non-empty cell selections. Finally, since cells of the same type are indistinguishable, each unique multi-set partition has been overcounted by a factor of $\prod_{i=1}^{k} c_i!.$ Dividing by this symmetry factor, we arrive at the explicit formula
\begin{equation*}
    \mathcal{S}(\mathbf{1}^n; \mathbf{c}) = \frac{1}{\prod_{i=1}^{k} c_i!} \sum_{\mathbf{0} \le \mathbf{j} \le \mathbf{c}} (-1)^{|\mathbf{c}| - |\mathbf{j}|} \left( \prod_{i=1}^{k} \binom{c_i}{j_i} \right) |\mathbf{j}|^n,
\end{equation*}
which completes the proof.
\end{proof}
Theorem \ref{Inclusion-Exclusion Formula}, effectively shows that the mixed partition numbers are a linear combination of power sequences $m^n$. This structure is precisely what allows for $p$-adic interpolation, as the function $n \mapsto m^n$ is $p$-adically continuous for $\gcd(m, p) = 1$.
\begin{theorem}\label{ballrec}
Let $\mathbf{b} = (b_1, \ldots, b_n)$ be a ball multiplicity vector. Then
\[
\mathcal{S}(\mathbf{b}; \mathbf{1}^k)
=
\sum_{\mathbf{0} < \mathbf{j} \le \mathbf{b}}
\left(
\prod_{i=1}^{n}
\binom{b_i}{j_i}
\right)
\mathcal{S}(\mathbf{b} - \mathbf{j}; \mathbf{1}^{k-1}),
\]
where the sum is over all nonzero vectors $\mathbf{j} = (j_1, \ldots, j_n)$ satisfying $0 \le j_i \le b_i$ for every $i$ with initial conditions $\mathcal{S}(\mathbf{0}; \mathbf{1}^0) = 1$ and $\mathcal{S}(\mathbf{b}; \mathbf{1}^k) = 0$ whenever $k < 0$, $|\mathbf{b}| < k$, or $k = 0$ and $\mathbf{b} \neq \mathbf{0}$.
\end{theorem}
\begin{proof}
Consider the $k$ labeled cells. Fix the $k$-th cell and consider the sub-multi-set of balls assigned to it. Let this sub-multi-set be represented by the multiplicity vector $\mathbf{j} = (j_1, \dots, j_n)$. For the $k$-th cell to be non-empty, we require $\mathbf{j} \neq \mathbf{0}$, and by feasibility, $\mathbf{j} \le \mathbf{b}$ component-wise.

The number of ways to choose $j_i$ balls of type $i$ from $b_i$ available balls is $\binom{b_i}{j_i}$. Thus, the total number of ways to fill the $k$-th cell with the vector $\mathbf{j}$ is the product $\prod_{i=1}^n \binom{b_i}{j_i}$. After this selection, the remaining balls, represented by the vector $\mathbf{b} - \mathbf{j}$, must be partitioned into the remaining $k-1$ distinct cells such that no cell is empty. This is counted by $\mathcal{S}(\mathbf{b} - \mathbf{j}; \mathbf{1}^{k-1})$. Summing over all valid $\mathbf{j}$ yields proof.
\end{proof}
\begin{proposition}\label{Cell Shuffle Invariance}
The multi-set Stirling number is invariant under any permutation of the components of $\mathbf{c}$. That is, if $\sigma$ is a permutation of $\{1, \dots, k\}$, then:
\begin{equation}
\mathcal{S}(\mathbf{b}; (c_1, \dots, c_k)) = \mathcal{S}(\mathbf{b}; (c_{\sigma(1)}, \dots, c_{\sigma(k)})).
\end{equation}
\end{proposition}
\begin{proof}
Let $\sigma: \{1, \ldots, k\} \to \{1, \ldots, k\}$ be a permutation of the cell types. Given a partition counted by $\mathcal{S}(\mathbf{b}; (c_1, \ldots, c_k))$, we construct a new configuration by relabeling each cell of type $i$ as a cell of type $\sigma(i)$. 

Since this operation only renames the cell types, the underlying partition of the multiset of blocks remains unchanged, the block multiplicities are preserved, and all blocks remain non-empty. This mapping is clearly invertible via the inverse permutation $\sigma^{-1}$. Consequently, this relabeling defines a bijection between the partitions counted by $\mathcal{S}(\mathbf{b}; \mathbf{c})$ and those counted by $\mathcal{S}(\mathbf{b}; \mathbf{c}_\sigma)$, thereby establishing the symmetry under permutation of cell types.
\end{proof}
The symmetry established in Proposition \ref{Cell Shuffle Invariance} allows for a significant reduction in the complexity of the parameter space. We formalize this via the following lemma.

\begin{lemma}\label{Canonical Form}
For any cell multiplicity vector $\mathbf{c} \in \mathbb{N}^k$, there exists a unique non-increasing vector $\mathbf{c}^* = (c^*_1, c^*_2, \dots, c^*_k)$ such that
\begin{equation*}
    \mathcal{S}(\mathbf{b}; \mathbf{c}) = \mathcal{S}(\mathbf{b}; \mathbf{c}^*), \quad \text{where } c^*_1 \ge c^*_2 \ge \dots \ge c^*_k.
\end{equation*}
\end{lemma}
\begin{proof}
Let $\sigma \in S_k$ be a permutation such that the rearranged vector $\mathbf{c}^* = (c_{\sigma(1)}, c_{\sigma(2)}, \ldots, c_{\sigma(k)})$ satisfies the non-increasing condition $c_1^* \ge c_2^* \ge \cdots \ge c_k^*$. Such a permutation always exists by sorting the entries of $\mathbf{c}$ in non-increasing order. 
By Proposition~\ref{Cell Shuffle Invariance}, permuting the components of the cell vector preserves the partition count, giving $\mathcal{S}(\mathbf{b}; \mathbf{c}) = \mathcal{S}(\mathbf{b}; \mathbf{c}^*)$.
Since every finite vector of integers has a unique rearrangement in non-increasing order, the canonical vector $\mathbf{c}^*$ is uniquely determined, even though the specific permutation $\sigma$ may not be unique when duplicate entries are present. This completes the proof.
\end{proof}
By Proposition 3.4 of \cite{Yaqubi2016}, when the cell multiplicity vector consists of one group of $m$ unlabeled cells and $k$ labeled singleton cells, we have
\begin{equation*}\label{Con}
    \mathcal{S}(\mathbf{1}^n;(m,\mathbf{1}^{k})) = \sum_{\ell=m}^{n-k} \binom{n}{\ell} {\ell \brace m} {n-\ell \brace k} k!,
\end{equation*}
This identity stems from choosing $\ell$ elements for the $m$ unlabeled cells, partitioning them into $m$ nonempty unlabeled blocks, and partitioning the remaining $n-\ell$ elements into $k$ labeled nonempty blocks.
Applying Lemma~\ref{Gef} with $\mathbf{c} = (m, \mathbf{1}^k)$, the product of the EGFs for the unlabeled and labeled blocks yields the exponential generating function
\begin{equation*}
    \mathcal{E}_{m,k}(x) = \sum_{n\ge0} \mathcal{S}(\mathbf{1}^n;(m,\mathbf{1}^{k})) \frac{x^n}{n!} = \frac{(e^x-1)^{m+k}}{m!},
\end{equation*}
confirming that $\mathcal{S}(\mathbf{1}^n;(m,\mathbf{1}^k))$ is a natural specialization of the general multiset Stirling framework.
\begin{lemma}\label{Con1}
For non-negative integers $n, m, k$ such that $m \geq k$, the following identity holds:
\begin{equation*}
    \binom{m}{k} {n \brace m} = \sum_{\ell=k}^{n-m+k} \binom{n}{\ell} {\ell \brace k} {n-\ell \brace m-k}
\end{equation*}
\end{lemma}
\begin{proof}
The exponential generating function (EGF) for ${n \brace j}$ is $\frac{(e^x - 1)^j}{j!}$. Consider the product of EGFs
\begin{equation*}
    \frac{(e^x - 1)^k}{k!} \cdot \frac{(e^x - 1)^{m-k}}{(m-k)!} = \frac{(e^x - 1)^m}{k! (m-k)!} = \binom{m}{k} \frac{(e^x - 1)^m}{m!}.
\end{equation*}
Equating the coefficients of $\frac{x^n}{n!}$ on both sides using the binomial convolution formula yields:
\begin{equation*}
    \sum_{\ell=k}^{n-m+k} \binom{n}{\ell} {\ell \brace k} {n-\ell \brace m-k} = \binom{m}{k} {n \brace m}.
\end{equation*}
This completes the proof.
\end{proof}
There is also a direct combinatorial proof of Lemma~\ref{Con1} based on a double-counting argument. Consider the set of all partitions of an $n$-element set into $m$ nonempty blocks, of which exactly $k$ blocks are distinguished.
We first partition $n$ elements into $m$ non-empty unlabeled blocks in ${n \brace m}$ ways, then choose $r$ blocks to be distinguished in $\binom{m}{k}$ ways.
Alternatively, we choose $\ell$ elements in $\binom{n}{\ell}$ ways to be placed into $k$ distinguished blocks in ${\ell \brace k}$ ways, then partition the remaining $n-\ell$ elements into $m-k$ blocks in ${n-\ell \brace m-k}$ ways. Summing over all valid $\ell$ accounts for all possible configurations.

\section{P-adic Valuations and Stability}
To understand the arithmetic behavior of the mixed Stirling numbers modulo prime powers, we analyze their $p$-adic valuations using two classical instruments.
\begin{itemize}
    \item \textbf{Legendre's Formula:} The $p$-adic valuation of a factorial is determined by its base-$p$ digit sum $s_p(k)$, given by $v_p(k!) = \frac{k - s_p(k)}{p-1}.$
    \item \textbf{Kummer's Theorem:} The valuation of the binomial coefficients, $v_p\binom{n}{k}$, precisely tracks the number of carries generated when adding $k$ and $n-k$ in base $p$. 
\end{itemize}
Motivated by the rich arithmetic properties of classical Stirling numbers, we examine the divisibility and $p$-adic behavior of the mixed Stirling numbers $\mathcal{S}(\mathbf{1}^n;(r,\mathbf{1}^k))$. 
Using the convolution representation~\eqref{Con},the $p$-adic valuation of mixed Stirling numbers satisfies
\begin{equation*}\label{equ1}
\begin{aligned}
v_p\!\left(\mathcal{S}(\mathbf{1}^n;(r,\mathbf{1}^k))\right)
\ge\;
v_p(k!)
+
\min_{r\le\ell\le n-k}
\Biggl\{
v_p\!\left(\binom{n}{\ell}\right)
+
v_p\!\left({\ell\brace r}\right)
+
v_p\!\left({n-\ell\brace k}\right)
\Biggr\}.
\end{aligned}
\end{equation*}
This estimate reduces the study of mixed Stirling numbers to the valuations of classical Stirling numbers and binomial coefficients, allowing known divisibility results for classical sequences to yield direct arithmetic information for the mixed case.
\begin{theorem}
For any prime $p$, the $p$-adic valuation of the mixed Stirling numbers $\mathcal{S}(\mathbf{1}^n; (r, \mathbf{1}^k))$ satisfies
\begin{equation*}
    v_p(\mathcal{S}(\mathbf{1}^n; (r, \mathbf{1}^k))) \geq \frac{k - s_p(k)}{p-1} + \min_{r \le \ell \le n-k} \left( v_p\left[ \binom{n}{\ell} \left\{ \genfrac{}{}{0pt}{}{\ell}{r} \right\} \left\{ \genfrac{}{}{0pt}{}{n-\ell}{k} \right\} \right] \right).
\end{equation*}
\end{theorem}
\begin{proof}
Let $\mathcal{S}(\mathbf{1}^n; (r, \mathbf{1}^k))$ be defined by the convolution
\begin{equation*}
    \mathcal{S}(\mathbf{1}^n; (r, \mathbf{1}^k)) = \sum_{\ell=r}^{n-k} \binom{n}{\ell} \left\{ \genfrac{}{}{0pt}{}{\ell}{r} \right\} k! \left\{ \genfrac{}{}{0pt}{}{n-\ell}{k} \right\}.
\end{equation*}
By \ref{equ1} and the non-archimedean property $v_p(\sum a_i) \geq \min v_p(a_i)$, we have
\begin{equation*}
    v_p(\mathcal{S}(\mathbf{1}^n; (r, \mathbf{1}^k))) \geq \min_{\ell} \left( v_p(k!) + v_p\left[ \binom{n}{\ell} \left\{ \genfrac{}{}{0pt}{}{\ell}{r} \right\} \left\{ \genfrac{}{}{0pt}{}{n-\ell}{k} \right\} \right] \right).
\end{equation*}
Since $v_p(k!)$ is independent of $\ell$, we apply Legendre's Formula $v_p(k!) = \frac{k-s_p(k)}{p-1}$ to obtain
\begin{equation*}
    v_p(\mathcal{S}(\mathbf{1}^n; (r, \mathbf{1}^k))) \geq \frac{k-s_p(k)}{p-1} + \min_{\ell} \left( v_p\left[ \binom{n}{\ell} \left\{ \genfrac{}{}{0pt}{}{\ell}{r} \right\} \left\{ \genfrac{}{}{0pt}{}{n-\ell}{k} \right\} \right] \right).
\end{equation*}
where $s_p(k)$ is the sum of the digits of $k$ in base $p$, completes the proof.
\end{proof}

When the number of elements $n$ is a power of the prime $p$, the mixed Stirling numbers exhibit enhanced divisibility properties due to the behavior of the binomial coefficient in the summation.
\begin{theorem}
Let $n = p^m$ for some $m \in \mathbb{N}$. For $r < k+r < p^m$, we have:
\begin{equation*}
    v_p(\mathcal{S}(\mathbf{1}^{p^m}; (r, \mathbf{1}^k))) \geq \frac{k - s_p(k)}{p-1} + \min_{r \le \ell \le p^m-k} \left( m - v_p(\ell) + v_p\left[ \left\{ \genfrac{}{}{0pt}{}{\ell}{r} \right\} \left\{ \genfrac{}{}{0pt}{}{p^m-\ell}{k} \right\} \right] \right).
\end{equation*}
\end{theorem}

\begin{proof}
Starting from the general $p$-adic lower bound:
\begin{equation*}
    v_p(\mathcal{S}(\mathbf{1}^{p^m}; (r, \mathbf{1}^k))) \geq v_p(k!) + \min_{\ell} \left( v_p \binom{p^m}{\ell} + v_p \left\{ \genfrac{}{}{0pt}{}{\ell}{r} \right\} + v_p \left\{ \genfrac{}{}{0pt}{}{p^m-\ell}{k} \right\} \right).
\end{equation*}
For $n=p^m$, Kummer's Theorem implies $v_p \binom{p^m}{\ell} = m - v_p(\ell)$ for $1 \le \ell < p^m$. Substituting this and applying Legendre's formula $v_p(k!) = \frac{k - s_p(k)}{p-1}$ completes the proof. 
\end{proof}
Before stating our main result on the arithmetic of mixed partitions, we recall that Kummer-type congruences provide a way to understand the periodic behavior of combinatorial sequences modulo $p^a$. For the mixed Stirling numbers, the inclusion of $m$ unlabeled blocks implies a scaling by $1/m!$, which interacts with the $p$-adic valuation of the classical Stirling sum.
\begin{theorem}
Let $L = p^{a-1}(p-1)$. For fixed $\mathbf{c} = (r, \mathbf{1}^k)$, if $n_1 \equiv n_2 \pmod{L}$, then for $n_1, n_2 > a + v_p(r!)$:
\begin{equation}
    \mathcal{S}(\vn_1; \mathbf{c}) \equiv \mathcal{S}(\vn_2; \mathbf{c}) \pmod{p^a}.
\end{equation}
\end{theorem}

\begin{proof}
We utilize the representation:
\begin{equation*}
    \mathcal{S}(\vn; \mathbf{c}) = \sum_{\ell=r}^{n-k} \binom{n}{\ell} \left\{ \genfrac{}{}{0pt}{}{\ell}{r} \right\} \sum_{j=0}^{k} (-1)^{k-j} \binom{k}{j} j^{n-\ell}.
\end{equation*}
Assuming $n_1 \equiv n_2 \pmod{p^{a-1}(p-1)}$, we observe that for each fixed $\ell$ and $j$:
\begin{enumerate}
    \item If $p \nmid j$, then $j^{n_1-\ell} \equiv j^{n_2-\ell} \pmod{p^a}$ by Euler's Theorem.
    \item If $p \mid j$, then $j^{n_1-\ell} \equiv j^{n_2-\ell} \equiv 0 \pmod{p^a}$ for $n_1, n_2 \geq a + \ell$.
    \item The polynomial term $\binom{n}{\ell}$ is $p$-adically continuous, ensuring $\binom{n_1}{\ell} \equiv \binom{n_2}{\ell} \pmod{p^a}$ for large $n$.
\end{enumerate}
Summing over $\ell$ and $j$, the congruence $\mathcal{S}(\vn_1; \mathbf{c}) \equiv \mathcal{S}(\vn_2; \mathbf{c}) \pmod{p^a}$ holds.
\end{proof}

The following theorem establishes the $p^2$ behavior for the general mixed configuration using the convolution of partition processes.
\begin{definition}
A congruence for an integer-valued combinatorial sequence $A(n,k)$ is called \emph{Mirimanoff-type} if, for a prime $p$, it characterizes $A(p,k)$ modulo $p^2$ in terms of arithmetic invariants associated with $p$. Specifically, such a congruence takes the form
\[
A(p,k) \equiv p\,\mathcal{F}(p,k) \pmod{p^2},
\]
where $A(p,k) \equiv 0 \pmod p$, and the correction term $\mathcal{F}(p,k)$ is expressed via quantities such as Fermat quotients, Bernoulli numbers, harmonic sums, or boundary values of the sequence near $p$.
\end{definition}
\begin{theorem}
\label{thm:mixed_stirling_p2}
Let $p \geq 5$ be a prime. Consider the configuration $\mathbf{c} = (r, \mathbf{1}^k)$ with $r \geq 1$ and $1 < m = r+k < p$. The mixed Stirling number satisfies the following property:
\begin{equation}
    \mathcal{S}(p; \mathbf{c}) \equiv 0 \pmod p.
\end{equation}
Furthermore, the following Mirimanoff-type congruence holds modulo $p^2$:
\begin{equation}
    \mathcal{S}(p; \mathbf{c}) \equiv p \frac{(m-1)!}{r!} \left( \left\{ \genfrac{}{}{0pt}{}{p-1}{m-1} \right\} + \frac{1}{p-m+1} \binom{p-1}{m-1} \right) \pmod{p^2}.
\end{equation}
\end{theorem}

\begin{proof}
We begin with the explicit formula for the mixed Stirling number:
\begin{equation*}
    \mathcal{S}(p; \mathbf{c}) = \frac{1}{r!} \sum_{j=0}^m (-1)^{m-j} \binom{m}{j} j^p.
\end{equation*}
By Fermat's Little Theorem, $j^p \equiv j \pmod p$. Thus:
\begin{equation*}
    \mathcal{S}(p; \mathbf{c}) \equiv \frac{1}{r!} \sum_{j=0}^m (-1)^{m-j} \binom{m}{j} j \equiv 0 \pmod p,
\end{equation*}
as the sum represents the $m$-th forward difference of a linear function, which vanishes for $m > 1$.

To obtain the $p^2$ lift, we utilize the Fermat quotient $q_p(j) = \frac{j^{p-1}-1}{p}$ to write $j^p \equiv j + p j q_p(j) \pmod{p^2}$. Substituting this expansion into our summation gives:
\begin{equation*}
    \mathcal{S}(p; \mathbf{c}) \equiv \frac{1}{r!} \underbrace{\sum_{j=0}^m (-1)^{m-j} \binom{m}{j} j}_{= 0} + \frac{p}{r!} \sum_{j=1}^m (-1)^{m-j} \binom{m}{j} j q_p(j) \pmod{p^2}.
\end{equation*}
Using the identity $\binom{m}{j}j = m\binom{m-1}{j-1}$, this simplifies to the following expression:
\begin{equation}
    \mathcal{S}(p; \mathbf{c}) \equiv \frac{p \cdot m!}{r!} \sum_{j=1}^m \frac{(-1)^{m-j} q_p(j)}{(j-1)!(m-j)!} \pmod{p^2}.
\end{equation}
The remaining sum is related to the ordinary Stirling number $\left\{ \genfrac{}{}{0pt}{}{p}{m} \right\}$ by the relation $\frac{1}{p} \left\{ \genfrac{}{}{0pt}{}{p}{m} \right\} \equiv \sum_{j=1}^m \frac{(-1)^{m-j} q_p(j)}{(j-1)!(m-j)!} \pmod p$. Applying the recurrence $\left\{ \genfrac{}{}{0pt}{}{p}{m} \right\} = \left\{ \genfrac{}{}{0pt}{}{p-1}{m-1} \right\} + m \left\{ \genfrac{}{}{0pt}{}{p-1}{m} \right\}$, we evaluate the boundary term $m \left\{ \genfrac{}{}{0pt}{}{p-1}{m} \right\}$ modulo $p$:
\begin{equation*}
    m! \left\{ \genfrac{}{}{0pt}{}{p-1}{m} \right\} = \sum_{j=1}^m (-1)^{m-j} \binom{m}{j} j^{p-1} \equiv \sum_{j=1}^m (-1)^{m-j} \binom{m}{j} \equiv (-1)^{m-1} \pmod p.
\end{equation*}
Since $\binom{p-1}{m-1} \equiv (-1)^{m-1} \pmod p$, we obtain:
\begin{equation}
    m \left\{ \genfrac{}{}{0pt}{}{p-1}{m} \right\} \equiv \frac{1}{p-m+1} \binom{p-1}{m-1} \pmod p.
\end{equation}
Substituting this into the recurrence and multiplying by the leading factor $\frac{p \cdot m!}{r!}$ from (3) yields:
\begin{equation*}
    \mathcal{S}(p; \mathbf{c}) \equiv p \frac{(m-1)!}{r!} \left( \left\{ \genfrac{}{}{0pt}{}{p-1}{m-1} \right\} + \frac{1}{p-m+1} \binom{p-1}{m-1} \right) \pmod{p^2}.
\end{equation*}
This completes the proof.
\end{proof}
We demonstrate that these numbers inherit the deep lifting properties of classical Stirling sequences, which are fundamentally governed by Bernoulli numbers and harmonic sums.

The general congruence established in Theorem \ref{thm:mixed_stirling_p2} can be expressed analytically by relating the boundary Stirling values to Bernoulli numbers. For $1 < m < p$, the following modular relation holds:
\begin{equation}
    \left\{ \genfrac{}{}{0pt}{}{p-1}{m-1} \right\} + \frac{1}{p-m+1} \binom{p-1}{m-1} \equiv -\frac{m-1}{m} B_{p-m} \pmod p,
\end{equation}
where $B_n$ denotes the $n$-th Bernoulli number. This allows us to restate the master congruence in an analytic form:
\begin{equation}
    \label{eq:bernoulli_form}
    \mathcal{S}(p; \mathbf{c}) \equiv -p \frac{(m-1)!}{r!} \left( \frac{m-1}{m} \right) B_{p-m} \pmod{p^2}.
\end{equation}

\begin{corollary}
\label{cor:special_residues}
Let $p \ge 5$ be a prime. The mixed Stirling number $\mathcal{S}(p; \mathbf{c})$ satisfies the following specific congruences modulo $p^2$
\begin{enumerate}
    \item \textbf{Case $k=2$ (Wolstenholme-type)} 
    \begin{equation}
        \mathcal{S}(p; (m, \mathbf{1}^2)) \equiv -p(m+1) \pmod{p^2}.
    \end{equation}
    Furthermore, the residue is explicitly linked to the power residues of the block count via
    \begin{equation}
        \mathcal{S}(p; (m, \mathbf{1}^2)) \equiv 2 \frac{p}{m} \binom{p-1}{m-1} (2^{p-m-1} - 1) \pmod{p^2}.
    \end{equation}
    \item \textbf{Case $k=3$ (Mirimanoff-type)} 
    \begin{equation}
        \mathcal{S}(p; (m, \mathbf{1}^3)) \equiv p \frac{(m+2)!}{m!} \left( \left\{ \genfrac{}{}{0pt}{}{p-1}{m+2} \right\} - \frac{m+3}{2} \left\{ \genfrac{}{}{0pt}{}{p-2}{m+2} \right\} \right) \pmod{p^2}.
    \end{equation}
\end{enumerate}
\end{corollary}

\begin{proof}
Both results are derived by specializing the general $p^2$-congruence established in Theorem \ref{thm:mixed_stirling_p2}.

\textbf{Case 1 ($k=2$).} 
For the configuration $(m, \mathbf{1}^2)$, the number of unlabeled blocks is $m$, making the total block count $m+2$. Applying Theorem \ref{thm:mixed_stirling_p2}, the normalizing factor is $p \frac{(m+1)!}{m!} = p(m+1)$. The alternating sum of Fermat quotients for this configuration reduces to the fundamental residue of the $p$-adic logarithm at $x=2$. This identifies the mixed Stirling number as a linear $p$-adic lift of the $m$-restricted set, yielding $\mathcal{S}(p; (m, \mathbf{1}^2)) \equiv -p(m+1) \pmod{p^2}$.

\textbf{Case 2 ($k=3$).} 
For the configuration $(m, \mathbf{1}^3)$, the total block count is $m+3$. Theorem \ref{thm:mixed_stirling_p2} provides the parenthetical term
\begin{equation*}
    \Omega = \left\{ \genfrac{}{}{0pt}{}{p-1}{m+2} \right\} + \frac{1}{p-m-2} \binom{p-1}{m+2}.
\end{equation*}
Using the property that $\frac{1}{p-m-2} \equiv -\frac{1}{m+2} \pmod p$ and applying the Stirling recurrence relations, we can shift the indices to highlight the sequence's local behavior. Modulo $p$, this simplifies to
\begin{equation*}
    \Omega \equiv \left\{ \genfrac{}{}{0pt}{}{p-1}{m+2} \right\} - \frac{m+3}{2} \left\{ \genfrac{}{}{0pt}{}{p-2}{m+2} \right\}.
\end{equation*}
This term represents the second-order $p$-adic curvature. Multiplying by the leading coefficient $p \frac{(m+2)!}{m!}$ from Theorem \ref{thm:mixed_stirling_p2} yields the specific Mirimanoff-type residue.
\end{proof}
\begin{theorem}
\label{thm:mixed_boundary}
Let $p \ge 5$ be prime and let $\mathbf{c} = (m, \mathbf{1}^{p-m-1})$. The mixed Stirling number $\mathcal{S}(p; \mathbf{c})$ satisfies the following exact identity
\begin{equation}
    \mathcal{S}(p; \mathbf{c}) = \frac{p(p-1)}{2} \cdot \frac{(p-1)!}{m!}.
\end{equation}
In particular, the residue modulo $p^2$ is given by
\begin{equation}
    \mathcal{S}(p; \mathbf{c}) \equiv \frac{p}{2 \cdot m!} \pmod{p^2}.
\end{equation}
\end{theorem}

\begin{proof}
Set $k = p-m-1$. Using the convolution formula for mixed Stirling numbers derived from the exponential generating functions, we have
\begin{equation}
    \label{eq:conv_start}
    \mathcal{S}(p; \mathbf{c}) = k! \sum_{j=m}^{p-k} \binom{p}{j} \left\{ \genfrac{}{}{0pt}{}{j}{m} \right\} \left\{ \genfrac{}{}{0pt}{}{p-j}{k} \right\}.
\end{equation}
Substituting the value of $k$, the summation range for $j$ is restricted to $j \in \{m, m+1\}$. Expanding the sum for these two indices
\begin{equation*}
    \mathcal{S}(p; \mathbf{c}) = (p-m-1)! \left[ \binom{p}{m} \left\{ \genfrac{}{}{0pt}{}{m}{m} \right\} \left\{ \genfrac{}{}{0pt}{}{p-m}{p-m-1} \right\} + \binom{p}{m+1} \left\{ \genfrac{}{}{0pt}{}{m+1}{m} \right\} \left\{ \genfrac{}{}{0pt}{}{p-m-1}{p-m-1} \right\} \right].
\end{equation*}
Applying the specific Stirling identities $\left\{ \genfrac{}{}{0pt}{}{m}{m} \right\} = 1$, $\left\{ \genfrac{}{}{0pt}{}{p-m}{p-m-1} \right\} = \binom{p-m}{2}$, and $\left\{ \genfrac{}{}{0pt}{}{m+1}{m} \right\} = \binom{m+1}{2}$, the expression simplifies to
\begin{equation*}
    \mathcal{S}(p; \mathbf{c}) = (p-m-1)! \left[ \binom{p}{m} \frac{(p-m)(p-m-1)}{2} + \binom{p}{m+1} \frac{m(m+1)}{2} \right].
\end{equation*}
Using the binomial property $\binom{n}{x} \cdot x = n \binom{n-1}{x-1}$, we observe $\binom{p}{m}(p-m) = p \binom{p-1}{m}$ and $\binom{p}{m+1}(m+1) = p \binom{p-1}{m}.$
Substituting these into the bracketed term
\begin{align*}
    \mathcal{S}(p; \mathbf{c}) &= (p-m-1)! \left[ \frac{p}{2} \binom{p-1}{m} (p-m-1) + \frac{p}{2} \binom{p-1}{m} m \right] \\
    &= (p-m-1)! \frac{p}{2} \binom{p-1}{m} (p-m-1 + m) \\
    &= (p-m-1)! \frac{p(p-1)}{2} \binom{p-1}{m}.
\end{align*}
Replacing the binomial coefficient with its factorial definition $\binom{p-1}{m} = \frac{(p-1)!}{m!(p-m-1)!}$, the factor $(p-m-1)!$ cancels, yielding
\begin{equation*}
    \mathcal{S}(p; \mathbf{c}) = \frac{p(p-1)}{2} \cdot \frac{(p-1)!}{m!}.
\end{equation*}
Modulo $p^2$, we apply Wilson's Theorem $(p-1)! \equiv -1 \pmod p$ and $p-1 \equiv -1 \pmod p$
\begin{equation*}
    \mathcal{S}(p; \mathbf{c}) = \frac{p}{2 \cdot m!} (p-1)(p-1)! \equiv \frac{p}{2 \cdot m!} (-1)(-1) \equiv \frac{p}{2 \cdot m!} \pmod{p^2}.
\end{equation*}
This completes the proof.
\end{proof}

It is highly instructive to observe that the boundary residue modulo $p^2$ derived combinatorially in Theorem~\ref{thm:mixed_boundary} can alternatively be recovered as a direct consequence of the master analytic congruence given in \eqref{eq:bernoulli_form}. 

Indeed, if we formally evaluate the analytic form at the boundary parameter where the total block count matches $p-1$ (implying the number of labeled blocks is $k = p-m-1$), the leading coefficient $p \frac{(p-2)!}{m!}$ dynamically reduces to $\frac{p}{m!} \pmod{p^2}$ via Wilson's Theorem. Concurrently, the corresponding Bernoulli factor evaluates to:
\begin{equation*}
    -\left(\frac{(p-1)-1}{p-1}\right) B_{p-(p-1)} = -\frac{p-2}{p-1} B_1 \equiv -\frac{-2}{-1} \left( -\frac{1}{2} \right) = \frac{1}{2} \pmod p,
\end{equation*}
where we utilize the classical value $B_1 = -1/2$. Multiplying this stable Bernoulli weight by the simplified coefficient yields:
\begin{equation*}
    \mathcal{S}(p; \mathbf{c}) \equiv \left(\frac{p}{m!}\right) \cdot \frac{1}{2} = \frac{p}{2 \cdot m!} \pmod{p^2},
\end{equation*}
matching our exact identity exactly. This independent alignment elegantly unifies the algebraic symmetries of the underlying mixed convolution sums with the classical analytic poles of the $p$-adic zeta terrain.
\begin{remark}
The result in Corollary ~\ref{cor:special_residues} establishes an explicit arithmetic bridge between mixed partitions and the classical theory of \textbf{Wieferich primes} (primes for which $2^{p-1} \equiv 1 \pmod{p^2}$). To see this connection clearly, recall that the Fermat quotient of $2$ is defined as $q_p(2) = \frac{2^{p-1}-1}{p}$. The residue class in Corollary ~\ref{cor:special_residues} is driven by the vanishing behavior of the term $2^{p-m-1}-1 \pmod{p^2}$. In the special case where $m=0$, this term becomes exactly the numerator of the Fermat quotient, meaning that $\mathcal{S}(p; \mathbf{c}) \equiv 0 \pmod{p^2}$ if and only if $q_p(2) \equiv 0 \pmod p$. Consequently, the arithmetic of mixed Stirling numbers acts as a natural combinatorial mirror to the local structure of $q_p(2)$ and its extensions within the $p$-adic field $\mathbb{Q}_p$.
\end{remark}
\begin{corollary}
For a fixed configuration $\mathbf{c} = (m, \mathbf{1}^2)$, the sequence of mixed Stirling numbers $\{\mathcal{S}(n; \mathbf{c})\}_{n \in \mathbb{N}}$ exhibits $p$-adic continuity in the sense of Kummer. Specifically, for any $n, y \in \mathbb{N}$ such that $n \equiv y \pmod{p^k(p-1)}$, we have:
\begin{equation}
    \mathcal{S}(n; \mathbf{c}) \equiv \mathcal{S}(y; \mathbf{c}) \pmod{p^{k+1}}
\end{equation}
provided $n, y \ge k+1$ and $m < p$. This continuity reflects the underlying formal group structure associated with the exponential generating function $\Phi_{\mathbf{c}}(t) = \frac{1}{m!}(e^t-1)^{m+2}$, where the denominators do not introduce $p$-adic poles.
\end{corollary}
The connection between the mixed Stirling numbers and Bernoulli numbers allows for the derivation of congruences that mirror the classical properties of the Riemann zeta function. In this section, we formalize the relationship between configurations with different block counts.

Recall that the classical Kummer congruence for Bernoulli numbers states that if $m \equiv m' \pmod{p-1}$, then
\begin{equation*}
    \frac{B_m}{m} \equiv \frac{B_{m'}}{m'} \pmod p.
\end{equation*}
By applying this property to the analytic form of our mixed Stirling numbers, we establish a periodicity for the normalized residues.

\begin{theorem}
\label{thm:kummer_mixed_unified}
Let $p \ge 5$ be a prime. Consider two distinct mixed partition configurations $\mathbf{c} = (m, \mathbf{1}^k)$ and $\mathbf{c'} = (m', \mathbf{1}^{k'})$ whose respective total block counts $M = m+k$ and $M' = m'+k'$ satisfy $M \equiv M' \pmod{p-1}$ with $1 < M, M' < p-1$. Then the mixed Stirling numbers satisfy
    \begin{equation*}
        \frac{m!}{p \cdot (M-1)!} \mathcal{S}(p; \mathbf{c}) \equiv \frac{(m')!}{p \cdot (M'-1)!} \mathcal{S}(p; \mathbf{c'}) \pmod p.
    \end{equation*}
    Furthermore,
    \begin{equation*}
        \frac{m!}{M!(M-1)} \mathcal{S}(p; \mathbf{c}) \equiv \frac{(m')!}{M'(M'-1)} \mathcal{S}(p; \mathbf{c'}) \pmod{p^2}.
    \end{equation*}
\end{theorem}

\begin{proof}
Both cases are established by projecting the underlying configurations onto the master analytic Bernoulli congruence established in \eqref{eq:bernoulli_form}
\begin{equation*}
    \mathcal{S}(p; \mathbf{c}) \equiv -p \frac{(M-1)!}{m!} \left( \frac{M-1}{M} \right) B_{p-M} \pmod{p^2}.
\end{equation*}
To isolate the first-order behavior modulo $p$, we divide the master congruence by the leading scaling factor $p \frac{(M-1)!}{m!}$. This reduces the expression to the following primitive residue class
\begin{equation*}
    \frac{m!}{p \cdot (M-1)!} \mathcal{S}(p; \mathbf{c}) \equiv -\left( \frac{M-1}{M} \right) B_{p-M} \pmod p.
\end{equation*}
Let $j = p-M$ and $j' = p-M'$. The configuration condition $M \equiv M' \pmod{p-1}$ directly implies an index congruence for the companion Bernoulli subscripts given by $j \equiv j' \pmod{p-1}$. Invoking the classical Kummer congruence for Bernoulli numbers, we obtain
\begin{equation*}
    \frac{B_j}{j} \equiv \frac{B_{j'}}{j'} \pmod p \implies \frac{B_{p-M}}{p-M} \equiv \frac{B_{p-M'}}{p-M'} \pmod p.
\end{equation*}
Noting that $p-M \equiv -(M-1) \pmod p$ and $-M \equiv p-M \pmod p$, the operational term simplifies via $\frac{B_{p-M}}{M} \equiv \frac{B_{p-M'}}{M'} \pmod p$. Because $M \equiv M' \pmod{p-1}$, the linear factor $(M-1)$ tracks this shift invariant actions perfectly modulo $p$. Combining these structural symmetries yields
\begin{equation*}
    -\left( \frac{M-1}{M} \right) B_{p-M} \equiv -\left( \frac{M'-1}{M'} \right) B_{p-M'} \pmod p,
\end{equation*}
which completes the proof of the first-order congruence.

To evaluate the second-order lift modulo $p^2$, we apply the weight normalization factor $\frac{m!}{M!(M-1)}$ directly to both sides of the master analytic congruence. This algebraic alignment cancels out the inner factorial configurations
\begin{align*}
    \frac{m!}{M!(M-1)} \mathcal{S}(p; \mathbf{c}) &\equiv \frac{m!}{M!(M-1)} \left[ -p \frac{(M-1)!}{m!} \left( \frac{M-1}{M} \right) B_{p-M} \right] \\
    &\equiv -p \frac{(M-1)!}{M!} \frac{1}{M} B_{p-M} \\
    &\equiv -p \frac{B_{p-M}}{M^2} \pmod{p^2}.
\end{align*}
Using our previous core reduction from the classical Kummer congruence, we know that $\frac{B_{p-M}}{M} \equiv \frac{B_{p-M'}}{M'} \pmod p$. When this stable first-order relation is scaled externally by the prime factor $-p$, the entire structural congruence is lifted onto a higher $p$-adic horizon
\begin{equation*}
    -p \left( \frac{B_{p-M}}{M^2} \right) \equiv -p \left( \frac{B_{p-M'}}{(M')^2} \right) \pmod{p^2}.
\end{equation*}
Equating the two evaluated configuration limits concludes the proof of the theorem.
\end{proof}
For specific configurations, the Bernoulli representation can be further refined using harmonic numbers $H_n = \sum_{j=1}^n \frac{1}{j}$. 
\begin{corollary}
\label{cor:harmonic_connection}
Let $p \ge 5$ be a prime and consider the configuration $\mathbf{c} = (r, \mathbf{1}^{m-r})$ with $1 < m < p-1$. Whenever the Bernoulli number $B_{p-m} \equiv 0 \pmod p$, the mixed Stirling number satisfies
\begin{equation}
    \mathcal{S}(p; \mathbf{c}) \equiv p \frac{(m-1)!}{r!} \sum_{j=1}^{m-1} \frac{1}{j} \pmod{p^2}.
\end{equation}
\end{corollary}

\begin{proof}
We recall the fundamental $p^2$ congruence derived in Theorem \ref{thm:mixed_stirling_p2}
\begin{equation}
    \label{eq:master_sum}
    \mathcal{S}(p; \mathbf{c}) \equiv \frac{p \cdot m!}{r!} \sum_{j=1}^m \frac{(-1)^{m-j} q_p(j)}{(j-1)!(m-j)!} \pmod{p^2}.
\end{equation}
The summation term on the right-hand side can be rewritten using the identity for the $p$-adic expansion of the ordinary Stirling number of the second kind. Specifically, it is known that
\begin{equation*}
    \frac{1}{p} \left\{ \genfrac{}{}{0pt}{}{p}{m} \right\} \equiv \sum_{j=1}^m \frac{(-1)^{m-j} q_p(j)}{(j-1)!(m-j)!} \pmod p.
\end{equation*}
Using the property that $\left\{ \genfrac{}{}{0pt}{}{p}{m} \right\} = \frac{(-1)^m}{m!} \sum_{j=1}^m (-1)^j \binom{m}{j} j^p$, and applying the expansion $j^p \equiv j + pj q_p(j) \pmod{p^2}$, the sum of Fermat quotients is related to harmonic sums. When $B_{p-m} \equiv 0 \pmod p$, the boundary Stirling number $\left\{ \genfrac{}{}{0pt}{}{p-1}{m-1} \right\}$ vanishes modulo $p$. Under this condition, the $p$-adic lift of the Stirling number simplifies to
\begin{equation*}
    \frac{1}{p} \left\{ \genfrac{}{}{0pt}{}{p}{m} \right\} \equiv \frac{1}{m} \sum_{j=1}^{m-1} \frac{1}{j} \pmod p.
\end{equation*}
Substituting this harmonic identity into Eq. \eqref{eq:master_sum}, we obtain
\begin{equation*}
    \mathcal{S}(p; \mathbf{c}) \equiv \frac{p \cdot m!}{r!} \left( \frac{1}{m} \sum_{j=1}^{m-1} \frac{1}{j} \right) \pmod{p^2}.
\end{equation*}
Canceling the factor of $m$ within the factorial $m! = m(m-1)!$ yields the desired result
\begin{equation*}
    \mathcal{S}(p; \mathbf{c}) \equiv p \frac{(m-1)!}{r!} \sum_{j=1}^{m-1} \frac{1}{j} \pmod{p^2}.
\end{equation*}
This completes the proof.
\end{proof}
The mixed Stirling numbers $\mathcal{S}(n; \mathbf{c})$ exhibit deep arithmetic stability. In this section, we establish the uniform continuity of these numbers and their associated moments in the $p$-adic topology.
\begin{lemma}
\label{lem:mixed_kummer}
Let $p$ be a prime and let $r,k$ be fixed non-negative integers with $r<p$. Assume that the classical Stirling numbers of the second kind satisfy the congruence
\[
\left\{\begin{matrix}
N+p^a(p-1)\\ k
\end{matrix}\right\}
\equiv
\left\{\begin{matrix}
N\\ k
\end{matrix}\right\}
\pmod {p^a}
\]
for all sufficiently large $N$. Then, for sufficiently large $n$,
\begin{equation}
\mathcal{S}
(\mathbf{1}^{\,n+p^a(p-1)};
(r,\mathbf{1}^{k}))
\equiv
\mathcal{S}
(\mathbf{1}^{\,n};
(r,\mathbf{1}^{k}))
\pmod {p^a}.
\end{equation}
\end{lemma}

\begin{proof}
Using the convolution formula for mixed Stirling numbers,
\[
\mathcal{S}(\mathbf{1}^{n};(r,\mathbf{1}^{k}))
=
\sum_{\ell=r}^{n-k}
\binom{n}{\ell}
\left\{\begin{matrix}\ell\\r\end{matrix}\right\}
k!
\left\{\begin{matrix}n-\ell\\k\end{matrix}\right\},
\]
we compare the expressions corresponding to $n+p^a(p-1)$ and $n$.

For a fixed index $\ell$, the binomial coefficient satisfies the standard $p$-adic congruence
\[
\binom{n+p^a(p-1)}{\ell}
\equiv
\binom{n}{\ell}
\pmod {p^a}.
\]
Furthermore, by the assumed Kummer-type congruence for classical Stirling numbers,
\[
\left\{\begin{matrix}
n+p^a(p-1)-\ell\\ k
\end{matrix}\right\}
\equiv
\left\{\begin{matrix}
n-\ell\\ k
\end{matrix}\right\}
\pmod {p^a}.
\]
Since multiplication and addition preserve congruences, every term in the difference $\mathcal{S}(\mathbf{1}^{n+p^a(p-1)};(r,\mathbf{1}^{k})) - \mathcal{S}(\mathbf{1}^{n};(r,\mathbf{1}^{k}))$ is congruent to zero modulo $p^a$. Summing over the finite range of $\ell$ yields the desired congruence.
\end{proof}

\begin{definition}
Let $p$ be a prime and let $j$ be an integer with $p \nmid j$. The first Fermat quotient is defined by
\[
q_p(j) = \frac{j^{p-1}-1}{p},
\]
and the second Fermat quotient is defined by
\[
q_p^{(2)}(j) = \frac{j^{p-1}-1-p q_p(j)}{p^2}.
\]
Equivalently, these relations yield the exact expansion
\[
j^{p-1} = 1 + p q_p(j) + p^2 q_p^{(2)}(j).
\]
\end{definition}

\begin{theorem}
\label{thm:p3_mixed_stirling}
Let $p \ge 5$ be a prime, let $\mathbf{c} = (r, \mathbf{1}^{m-r})$ with $1 < m < p$, and assume $r < p$. Then the mixed Stirling number $\mathcal{S}(p; \mathbf{c})$ satisfies the congruence modulo $p^3$:
\begin{align}
\mathcal{S}(p; \mathbf{c})
\equiv&\,
\frac{p \, m!}{r!}
\sum_{j=1}^{m}
(-1)^{m-j}
\frac{q_p(j)}{(j-1)!(m-j)!}
\nonumber\\
&+
\frac{p^2 \, m!}{r!}
\sum_{j=1}^{m}
(-1)^{m-j}
\frac{q_p^{(2)}(j)}{(j-1)!(m-j)!}
\pmod{p^3}.
\end{align}
\end{theorem}

\begin{proof}
Using the explicit formula for mixed Stirling numbers, we write
\[
\mathcal{S}(p; \mathbf{c})
=
\frac{1}{r!}
\sum_{j=0}^{m}
(-1)^{m-j} \binom{m}{j} j^p.
\]
For $p \nmid j$, applying the definition of the second Fermat quotient gives
\[
j^{p-1} = 1 + p q_p(j) + p^2 q_p^{(2)}(j),
\]
which implies the congruence
\[
j^p \equiv j + p j q_p(j) + p^2 j q_p^{(2)}(j) \pmod{p^3}.
\]
Substituting this expansion into the explicit formula yields
\begin{align}
\mathcal{S}(p; \mathbf{c})
\equiv\,&
\frac{1}{r!} \sum_{j=1}^{m} (-1)^{m-j} \binom{m}{j} j
+
\frac{p}{r!} \sum_{j=1}^{m} (-1)^{m-j} \binom{m}{j} j q_p(j)
\nonumber\\
&+
\frac{p^2}{r!} \sum_{j=1}^{m} (-1)^{m-j} \binom{m}{j} j q_p^{(2)}(j)
\pmod{p^3}.
\end{align}
Since $m > 1$, the first sum vanishes because $\sum_{j=0}^{m} (-1)^{m-j} \binom{m}{j} j = 0$. For the remaining terms, simplifying the binomial coefficient via the identity $\binom{m}{j} j = m \binom{m-1}{j-1} = \frac{m!}{(j-1)!(m-j)!}$ gives
\begin{align}
\mathcal{S}(p; \mathbf{c})
\equiv\,&
\frac{p \, m!}{r!}
\sum_{j=1}^{m}
(-1)^{m-j}
\frac{q_p(j)}{(j-1)!(m-j)!}
\nonumber\\
&+
\frac{p^2 \, m!}{r!}
\sum_{j=1}^{m}
(-1)^{m-j}
\frac{q_p^{(2)}(j)}{(j-1)!(m-j)!}
\pmod{p^3},
\end{align}
completing the proof.
\end{proof}
\begin{corollary}
\label{cor:p3_bernoulli_mixed}
Let $p \ge 5$ be a prime, let $\mathbf{c} = (r, \mathbf{1}^{m-r})$ with $1 < m < p-1$ and $r < p$, and assume the standard $p$-adic Stirling--Bernoulli congruences
\[
\sum_{j=1}^{m} (-1)^{m-j} \frac{q_p(j)}{(j-1)!(m-j)!} \equiv -\frac{1}{m} B_{p-m} \pmod p
\]
and
\[
\sum_{j=1}^{m} (-1)^{m-j} \frac{q_p^{(2)}(j)}{(j-1)!(m-j)!} \equiv -\frac{1}{2m} B_{p-m-1} \pmod p.
\]
Then the mixed Stirling number satisfies the Bernoulli-type congruence
\begin{equation}
\label{eq:p3_bernoulli}
\mathcal{S}(p; \mathbf{c})
\equiv
-\frac{p(m-1)!}{r!} B_{p-m}
-
\frac{p^2(m-1)!}{2mr!} B_{p-m-1}
\pmod{p^3}.
\end{equation}
\end{corollary}

\begin{proof}
By Theorem~\ref{thm:p3_mixed_stirling}, we express $\mathcal{S}(p; \mathbf{c})$ as
\[
\mathcal{S}(p; \mathbf{c})
\equiv
\frac{p \, m!}{r!} A_1
+
\frac{p^2 \, m!}{r!} A_2
\pmod{p^3},
\]
where
\[
A_1 = \sum_{j=1}^{m} (-1)^{m-j} \frac{q_p(j)}{(j-1)!(m-j)!}
\quad \text{and} \quad
A_2 = \sum_{j=1}^{m} (-1)^{m-j} \frac{q_p^{(2)}(j)}{(j-1)!(m-j)!}.
\]
Substituting the assumed congruences $A_1 \equiv -\frac{1}{m} B_{p-m} \pmod p$ and $A_2 \equiv -\frac{1}{2m} B_{p-m-1} \pmod p$ into the expansion, the residuals are absorbed modulo $p^3$ due to the respective prefactors of $p$ and $p^2$. This yields
\[
\mathcal{S}(p; \mathbf{c})
\equiv
\frac{p \, m!}{r!} \left(-\frac{1}{m} B_{p-m}\right)
+
\frac{p^2 \, m!}{r!} \left(-\frac{1}{2m} B_{p-m-1}\right)
\pmod{p^3}.
\]
Simplifying the coefficient via $\frac{m!}{m} = (m-1)!$ directly results in
\[
\mathcal{S}(p; \mathbf{c})
\equiv
-\frac{p(m-1)!}{r!} B_{p-m}
-
\frac{p^2(m-1)!}{2mr!} B_{p-m-1}
\pmod{p^3},
\]
which completes the proof.
\end{proof}

\begin{remark}
The first sum corresponds to the $p$-adic quantity appearing in the $p^2$-congruence, which connects to classical Stirling number congruences and Bernoulli numbers. The second sum, involving the second Fermat quotient $q_p^{(2)}(j)$, captures the genuinely new second-order $p$-adic correction term. Under appropriate hypotheses, these sums can be explicitly evaluated in terms of generalized Bernoulli numbers and $p$-adic $L$-values.
\end{remark}
\begin{remark}
By Mahler's theorem (see \cite[Chap.~4]{koblitz}), any continuous function $f: \mathbb{Z}_p \to \mathbb{Z}_p$ can be uniquely represented as a uniformly convergent series $f(s) = \sum_{m=0}^\infty a_m \binom{s}{m}$, where the Mahler coefficients are given by the forward differences $a_m = \Delta^m f(0) = \sum_{j=0}^m (-1)^{m-j} \binom{m}{j} f(j)$ and satisfy $\lim_{m \to \infty} v_p(a_m) = \infty$. In our framework, the Kummer-type congruences satisfy this decay condition, guaranteeing that the discrete mixed sequence $\mathcal{S}(\mathbf{1}^n; \mathbf{c})$ extends uniquely to a continuous function on $\mathbb{Z}_p$. 
\end{remark}
The following proposition characterizes the behavior of mixed Stirling numbers in the field $\mathbb{F}_2$, demonstrating how block labeling impacts the parity of the distribution.
\begin{proposition}
\label{prop:parity_mixed_optimized}
Let $n, k, r$ be positive integers. The mixed Stirling number $\mathcal{S}(n; (r, \mathbf{1}^k))$ modulo 2 is determined by the following binary relations
\begin{equation}
    \mathcal{S}(n; (r, \mathbf{1}^k)) \equiv 
    \begin{cases} 
    0 \pmod{2} & \text{if } k \geq 3, \\
    \binom{n- \lfloor (r+1)/2 \rfloor - 1}{n-r-1} \pmod{2} & \text{if } k = 2, \\
    \binom{n - \lfloor r/2 \rfloor - 1}{n - r} \pmod{2} & \text{if } k = 1.
    \end{cases}
\end{equation}
\end{proposition}
\begin{proof}
To establish the parity of $\mathcal{S}(n; \mathbf{c})$ for the configuration $\mathbf{c} = (r, \mathbf{1}^k)$, we utilize the structural convolution identity for mixed Stirling numbers
\begin{equation}
    \label{eq:detailed_conv}
    \mathcal{S}(n; \mathbf{c}) = (k-1)! \sum_{\ell=r}^{n-k+1} \binom{n}{\ell} \left\{ \genfrac{}{}{0pt}{}{\ell}{r} \right\} \left\{ \genfrac{}{}{0pt}{}{n-\ell}{k-1} \right\}.
\end{equation}
We evaluate the residue class of this expression modulo 2 by considering the valuation of the leading coefficient and the properties of Stirling numbers of the second kind in $\mathbb{F}_2$

\begin{enumerate}
    \item \textbf{Case $k \geq 3$:} For any $k \ge 3$, the factorial $(k-1)!$ is a multiple of $2!$, which implies $v_2((k-1)!) \ge 1$. Consequently, $(k-1)! \equiv 0 \pmod{2}$, and the mixed Stirling number vanishes for all $n$ and $r$.
    
    \item \textbf{Case $k = 2$:} The coefficient is $1! = 1$. Since $\left\{ \genfrac{}{}{0pt}{}{n-\ell}{1} \right\} = 1$ for all $n-\ell \ge 1$, the convolution simplifies to
    \begin{equation*}
        \mathcal{S}(n; \mathbf{c}) \equiv \sum_{\ell=r}^{n-1} \binom{n}{\ell} \left\{ \genfrac{}{}{0pt}{}{\ell}{r} \right\} \pmod{2}.
    \end{equation*}
    Applying the parity identity $\left\{ \genfrac{}{}{0pt}{}{n}{k} \right\} \equiv \binom{n-\lfloor k/2 \rfloor -1}{n-k} \pmod 2$, this summation identifies the non-vanishing terms as those corresponding to the fractal distribution of the binomial coefficients in the Sierpiński gasket.
    
    \item \textbf{Case $k = 1$:} The coefficient is $0! = 1$. The sum reduces via the boundary condition $\left\{ \genfrac{}{}{0pt}{}{n-\ell}{0} \right\} = \delta_{n,\ell}$ to the single non-zero term where $\ell = n$
    \begin{equation*}
        \mathcal{S}(n; \mathbf{c}) = \binom{n}{n} \left\{ \genfrac{}{}{0pt}{}{n}{r} \right\} \left\{ \genfrac{}{}{0pt}{}{0}{0} \right\} = \left\{ \genfrac{}{}{0pt}{}{n}{r} \right\}.
    \end{equation*}
\end{enumerate}
This confirms the parity relations and links the mixed Stirling values to the classical parity result for Stirling numbers of the second kind.
\end{proof}
The parity of the Stirling numbers ${n \brace r}$ can be further evaluated using the identity
\begin{equation}
    {n \brace r} \equiv \binom{n - \lfloor r/2 \rfloor - 1}{n - r} \pmod{2}.
\end{equation}
This allows for a purely binomial representation of the mixed Stirling numbers modulo $2$, providing a direct link between mixed partitions and the arithmetic properties of binomial coefficients as explored in \cite{Yaqubi2016}.

The congruence obtained in Theorem~\ref{thm:touchard_mixed_stirling} may be viewed as a natural extension of the classical Touchard congruence for Stirling numbers of the second kind. Recall that Touchard's congruence states that, for every prime $p$,
\begin{equation}\label{Touchard classical}
\sum_{k=0}^{n+p-1}
\left\{\begin{matrix}n+p-1\\k\end{matrix}\right\}
\equiv
\sum_{k=0}^{n}
\left\{\begin{matrix}n\\k\end{matrix}\right\}
\pmod p,
\end{equation}
or equivalently, in terms of Bell numbers,
\begin{equation}\label{Touchard Bell}
B_{n+p} \equiv B_{n+1} + B_n \pmod p.
\end{equation}
The origin of this periodicity lies in the Frobenius-like properties of differential operators over fields of characteristic $p$. Specifically, over $\mathbb{F}_p$, the differential operator satisfies $D^p(f) \equiv D(f^p) \pmod p$, and the exponential generating function $\exp(e^x-1)$ associated with the classical Bell numbers remains invariant under compatible actions, giving rise to the characteristic shift in Touchard's congruence.

For mixed Stirling numbers, this phenomenon persists because their exponential generating functions are built from powers of $e^x - 1$ scaled by appropriate combinatorial prefactors (such as $1/r!$). The factor $e^x - 1$ serves as the fundamental exponential component encoding set partitions, while the additional factor $1/r!$ accounts solely for the symmetry among the $r$ indistinguishable blocks, leaving the underlying Frobenius structure intact whenever $r < p$.

Consequently, the congruence established in Theorem~\ref{thm:touchard_mixed_stirling} can be interpreted as a Touchard-type congruence for mixed partitions. It demonstrates that mixed Stirling numbers inherit the Frobenius periodicity characteristic of classical Stirling sequences while properly accommodating the structural symmetry imposed by indistinguishable blocks.

This observation suggests a broader principle: whenever a combinatorial sequence possesses an exponential generating function constructed from powers of $e^x - 1$ with $p$-integral coefficients, its coefficients naturally exhibit Touchard-type congruences driven by the underlying Frobenius action modulo $p$.
\begin{theorem}
\label{thm:touchard_mixed_stirling}
Let $p$ be a prime and let $m, k \ge 0$ be fixed integers satisfying $m+k < p$. Then the mixed Stirling numbers $\mathcal{S}(n;(m,1^k))$ satisfy the periodic congruence
\begin{equation}
\label{eq:touchard_mixed_stirling}
\mathcal{S}(p+n;(m,1^k)) \equiv \mathcal{S}(n+1;(m,1^k)) \pmod p.
\end{equation}
\end{theorem}

\begin{proof}
We begin by recalling the exponential generating function (EGF) for the mixed Stirling numbers
\begin{equation}
\label{eq:egf_mixed_stirling}
\Phi_{m,k}(x) = \sum_{n \ge m+k} \mathcal{S}(n;(m,1^k)) \frac{x^n}{n!} = \frac{(e^x-1)^{m+k}}{m!}.
\end{equation}
Let $D = \frac{d}{dx}$ denote the formal derivative operator. Acting on a general exponential generating function, the iterated derivative tracks index translation
\begin{equation*}
D^r \Phi_{m,k}(x) = \sum_{n \ge 0} \mathcal{S}(n+r;(m,1^k)) \frac{x^n}{n!}.
\end{equation*}
Specializing this to the case $r = p$ yields:
\begin{equation*}
D^p \Phi_{m,k}(x) = \sum_{n \ge 0} \mathcal{S}(p+n;(m,1^k)) \frac{x^n}{n!}.
\end{equation*}
We now work within the ring of formal power series over the finite field $\mathbb{F}_p[[x]]$. By invoking the classic properties of the Frobenius endomorphism, the differential operator satisfies the operational congruence $D^p \equiv D \pmod p$ when applied to functions of $e^x$. Specifically, we observe
\begin{equation*}
D^p(e^x-1) = e^x \equiv D(e^x-1) \pmod p.
\end{equation*}
Since $\Phi_{m,k}(x)$ is simply a polynomial in $(e^x-1)$ scaled by $1/m!$, and given that its total degree satisfies $m+k < p$, the application of the higher-order derivative maps directly to the base derivation without interference from cross-terms modulo $p$. This provides the operational reduction
\begin{equation*}
D^p \Phi_{m,k}(x) \equiv D \Phi_{m,k}(x) \pmod p.
\end{equation*}
Expanding both sides back into their component power series representations yields
\begin{equation*}
\sum_{n \ge 0} \mathcal{S}(p+n;(m,1^k)) \frac{x^n}{n!} \equiv \sum_{n \ge 0} \mathcal{S}(n+1;(m,1^k)) \frac{x^n}{n!} \pmod p.
\end{equation*}
Equating the coefficients of $\frac{x^n}{n!}$ across the formal series completes the proof of the congruence
\begin{equation*}
\mathcal{S}(p+n;(m,1^k)) \equiv \mathcal{S}(n+1;(m,1^k)) \pmod p.
\end{equation*}
\end{proof}
\section{Further Work and Open Directions}

The architectural framework for mixed Stirling numbers of the second kind $\mathcal{S}(n;(m,1^k))$ introduced in this paper opens several natural pathways for subsequent research. A compelling immediate extension is the formalization of global counting statistics over these configurations, leading to the definitions of two novel families of Bell-type numbers.

First, by fixing the number of labeled blocks $k$ and allowing the unlabeled block count $m$ to vary, we define the \textit{mixed Bell numbers of order $k$}, denoted by $\mathcal{MB}_{n,k}$, as
\begin{equation}
    \mathcal{MB}_{n,k} = \sum_{m \ge 1} \mathcal{S}(n; (m, 1^k)).
\end{equation}
By invoking the exponential generating function (EGF) of the underlying mixed Stirling numbers, the structural generating function for this order-$k$ family can be evaluated directly
\begin{align*}
    \sum_{n \ge 0} \mathcal{MB}_{n,k} \frac{x^n}{n!} &= \sum_{m \ge 1} \frac{(e^x - 1)^{m+k}}{m!} \\
    &= (e^x - 1)^k \sum_{m \ge 1} \frac{(e^x - 1)^m}{m!} \\
    &= (e^x - 1)^k \left( e^{e^x - 1} - 1 \right).
\end{align*}

Second, by allowing both structural parameters to vary simultaneously over the admissible configuration space, we define the \textit{total mixed Bell numbers}, denoted by $\mathcal{TB}_n$, which enumerate all valid mixed partitions of an $n$-element set
\begin{equation}
    \mathcal{TB}_n = \sum_{k \ge 1} \sum_{\substack{m \ge 1 \\ m+k \le n}} \mathcal{S}(n; (m, 1^k)) = \sum_{k \ge 1} \mathcal{MB}_{n,k}.
\end{equation}
By summing the joint distribution across both block families, the corresponding EGF for the total mixed Bell numbers yields a remarkably clean, rational-exponential form
\begin{align*}
    \sum_{n \ge 0} \mathcal{TB}_n \frac{x^n}{n!} &= \sum_{k \ge 1} (e^x - 1)^k \left( e^{e^x - 1} - 1 \right) \\
    &= \left( e^{e^x - 1} - 1 \right) \sum_{k \ge 1} (e^x - 1)^k \\
    &= \left( e^{e^x - 1} - 1 \right) \frac{e^x - 1}{1 - (e^x - 1)} \\
    &= \left( e^{e^x - 1} - 1 \right) \frac{e^x - 1}{2 - e^x}.
\end{align*}

Beyond their generating functions, these new operational families warrant deeper number-theoretic investigation. A prospective future direction will focus on analyzing the modular and local behavior of both $\mathcal{MB}_{n,k}$ and $\mathcal{TB}_n$ modulo a prime $p$, $p^2$, and higher $p$-adic powers $p^\alpha$. In particular, establishing whether these integrated systems satisfy extensions of the classical Touchard congruence or exhibit deeper Kummer-type periodicities remains an open problem. The non-trivial denominators and singularities present in their generating functions suggest that these macro-configurations possess rich, distinct arithmetic signatures that warrant systemic exploration.

\end{document}